\documentclass[10 pt]{article}%
\usepackage{amsmath}
\usepackage{graphicx}
\usepackage{url}
\usepackage{amsfonts}
\usepackage{amssymb}%
\DeclareMathAlphabet{\cyrm}{U}{UWCyr}{m}{n}
\DeclareSymbolFont{cyrm}{U}{UWCyr}{m}{n}
\DeclareSymbolFontAlphabet{\cyrm}{cyrm}
\DeclareMathSymbol{\Evo}{\cyrm}{cyrm}{"03}
\newtheorem{theorem}{Theorem}[section]

\newtheorem{definition}[theorem]{Definition}
\newtheorem{example}[theorem]{Example}

\newtheorem{lemma}[theorem]{Lemma}

\newtheorem{proposition}[theorem]{Proposition}
\newtheorem{remark}[theorem]{Remark}

\newenvironment{proof}[1][Proof]{\noindent\textbf{#1.} }{\
\rule{0.5em}{0.5em}}
\begin{document}

\title{Normal forms for parabolic Monge-Amp\`{e}re equations}
\author{R. Alonso Blanco, G. Manno, F. Pugliese}
\maketitle

\begin{abstract}
We find normal forms for parabolic Monge-Amp\`{e}re equations. Of
these, the most general one holds for any equation admitting a
complete integral. Moreover, we explicitly give the determining
equation for such integrals; restricted to the analytic case, this
equation is shown to have solutions. The other normal forms
exhaust the different classes of parabolic Monge-Amp\`{e}re
equations with symmetry properties, namely, the existence of
classical or nonholonomic intermediate integrals. Our approach is
based on the equivalence between parabolic Monge-Amp\`{e}re
equations and particular distributions on a contact manifold, and
involves a classification of vector fields lying in the contact
structure. These are divided into three types and described in
terms of the simplest ones (characteristic fields of $1^{st}$
order PDE's).

\end{abstract}


\section{Introduction}

In the present paper we give a contribution to the problem of classifying
Monge-Amp\`{e}re equations (MAE) up to contact transformations. MAE's are
second order equations of the form%
\begin{equation}
N(z_{xx}z_{yy}-z_{xy}^{2})+Az_{xx}+Bz_{xy}+Cz_{yy}+D=0\text{,}
\label{MA_coordinate_bis}%
\end{equation}
in the unknown function $z=z(x,y)$, with coefficients $A$, $B$,
$C$, $D$, $N$ depending on $x$, $y$, $z$, $z_{x}$, $z_{y}$. As is
well known, any contact transformation maps a MAE into another
one. Therefore, a major problem concerning equations
(\ref{MA_coordinate_bis}) is their classification under the action
of the contact pseudogroup. An aspect of this problem consists in
finding \textit{normal forms}, i.e. some particularly simple model
equations, depending on functional parameters, such that any MAE
is locally contact equivalent to one and only one of them, for a
suitable choice of the parameters.

Below, we find normal forms of \textit{parabolic} MAE's, i.e. equations
(\ref{MA_coordinate_bis}) satisfying $B^{2}-4AC+4ND=0$. Geometrically, this
means that characteristic directions at any point of the $1$-jet bundle
$J^{1}(\tau)=\{(x,y,z,z_{x},z_{y})\}$ of the trivial bundle $\tau
:\mathbb{R}^{2}\times\mathbb{R}\rightarrow\mathbb{R}^{2}$ define a
$2$-dimensional subdistribution $\mathcal{D}$ of the contact distribution
$\mathcal{C}$:
\begin{equation}
\mathcal{C}=\{U=0\},\,\,\,\text{with}\,\,\,U=dz-z_{x}dx-z_{y}dy.
\label{Cartan_form}%
\end{equation}
As $\mathcal{D}$ is generally non integrable, it is necessary to consider also
its derived flag
\begin{equation}
\mathcal{D}\,\subset\,\mathcal{D}^{\prime}=\mathcal{D}+[\mathcal{D},\mathcal{D}]\,\subset\,\mathcal{D}^{\prime
\prime}=\mathcal{D}'+[\mathcal{D}',\mathcal{D}'] \label{derived flag}%
\end{equation}
whose properties allow to obtain important classification results
on parabolic MAE's in a simple and straightforward way. In fact,
such a study is based, to a large extent, on the geometry of
\textit{Cartan fields}, i.e. sections of $\mathcal{C}$. Quite
unexpectedly, generic Cartan fields are not contained in any
integrable $2$-dimensional subdistribution of $\mathcal{C}$. The
degree of \textquotedblleft genericity" of a Cartan field $X$ is
measured by a simple invariant, its \textit{type}: the higher the
type, the less symmetric $X$ is with respect to $\mathcal{C}$.
More precisely, $X$ is of type $2,3$ or $4$ if it is contained in
many, one or no integrable $2$-dimensional subdistribution of
$\mathcal{C}$, respectively (the operative definition of type is
given in section \ref{subsec_Type_Cartan_Field}).

\smallskip The main classification results in the present paper are summarized
by the following two theorems.

\begin{theorem}
\label{maintheor1} Let (\ref{MA_coordinate_bis}) be a parabolic MAE with
$C^{\infty}$ coefficients on some domain of $J^{1}(\tau)$. Then
(\ref{MA_coordinate_bis}) is locally contact equivalent to an equation of the
form%
\begin{equation}
z_{yy}-2az_{xy}+a^{2}z_{xx}=b\text{,\ } \label{Bryant_bis}%
\end{equation}
with $a,b\in C^{\infty}(J^{1}(\tau))$, if and only if it admits a complete integral.
\end{theorem}

Roughly speaking, a complete integral of (\ref{MA_coordinate_bis})
is any $3$-parametric family of solutions (see the more rigorous
Definition \ref{Integrale_Completo}). The existence of such a
family does not seem to be a strong condition on
(\ref{MA_coordinate_bis}); in fact, in section \ref{subsubsec:
A_rem_ex}, we provide a very large class of smooth parabolic MAE's
admitting a complete integral. Note that normal form
(\ref{Bryant_bis}) was proved to be true for every parabolic MAE
with \emph{real analytic coefficients} (\cite{Bryant Griffiths}).
In that paper, the proof essentially consisted in showing the
involutivity of a certain exterior differential system associated
with (\ref{MA_coordinate_bis}) and then applying Cartan-K\"{a}hler
theorem to such a system; below (Theorem \ref{Better_Than_Bryant})
we give an easier and more direct proof which makes use only of
Cauchy-Kovalevsky existence theorem. An immediate corollary of
this theorem and Theorem \ref{maintheor1} is the existence of a
complete integral for any real analytic parabolic MAE.

As it will be shown in section \ref{subsubsec:The_anal_case}, the existence of
a complete integral is equivalent to that of a \textit{generalized
intermediate integral}, i.e. a Cartan field of type less than $4$ contained in
$\mathcal{D}$. This is a generalization of both the classical (\cite{Goursat})
and the nonholonomic (\cite{KLR}) notion of intermediate integral; in fact, a
classical intermediate integral $f\in C^{\infty}(J^{1}(\tau))$ of
(\ref{MA_coordinate_bis}) can be identified with a \textit{hamiltonian field}
$X_{f}$ (a special kind of type $2$ Cartan field, see Definition
\ref{Hamiltonian_Field}) belonging to $\mathcal{D}$, while a nonholonomic
intermediate integral is any Cartan field of type $2$ in $\mathcal{D}$.

The existence of intermediate integrals of equation (\ref{MA_coordinate_bis})
is strictly linked to integrability properties of the derived flag
(\ref{derived flag}).

\begin{theorem}
\label{maintheor2} Let (\ref{MA_coordinate_bis}) be a parabolic MAE and
$\mathcal{D}$ be the corresponding characteristic distribution. If
$\dim\mathcal{D}^{\prime\prime}<5$ then equation (\ref{MA_coordinate_bis}) can
always be locally reduced by a contactomorphism to one of the following forms:

\begin{itemize}
\item[1)] $z_{yy}=0,$ when $\mathcal{D}$ is integrable;

\item[2)] $z_{yy}=b$, \ $b\in C^{\infty}(J^{1}(\tau))$,\ $\partial_{z_{x}}(b)\neq0$, when $\mathcal{D}%
^{\prime\prime}$ is $4$-dimensional and integrable;

\item[3)] $z_{yy}-2zz_{xy}+z^{2}z_{xx}=b$, \ $b\in
C^{\infty}(J^{1}(\tau))$
with $\partial_{z_{x}}(b)+z\partial_{z}(b)\neq0,$ \ when $\mathcal{D}%
^{\prime\prime}$ is $4$-dimensional and non integrable.
\end{itemize}
\end{theorem}

On the other hand, the three cases can be stated in terms of
intermediate integrals of (\ref{MA_coordinate_bis}), namely:

\begin{itemize}
\item[$1^\prime)$] There exist three (functionally independent)
intermediate integrals;

\item[$2^\prime)$] There exists just one intermediate integral;

\item[$3^\prime)$] There are no (classical) intermediate integrals
but there is exactly one nonholonomic intermediate integral.
\end{itemize}


The three normal forms of Theorem \ref{maintheor2} were already
known (\cite{Goursat,Tunitsky,Bryant Griffiths} respectively).
However, the alternative characterizations in terms of
intermediate integrals are original. Moreover, the conditions
given in \cite{Bryant Griffiths} for the validity of normal form
$3)$ and the relative proof are completely different and, in our
opinion, considerably more complicated and less transparent than
ours: in fact, it must be emphasized that our conditions are
easily computable for any given MAE.

\medskip The paper is structured as follows. In section \ref{sec.preliminary},
approximately following the approach of (\cite{KLR,Lychagin
contact}), the necessary preliminary notions on MAE's in the
framework of jet bundle formalism are given. Furthermore, the
equivalence between parabolic MAE's and lagrangian
subdistributions of $\mathcal{C}$ is explained.

Section \ref{sec:geom.cartan.fields} is devoted to the geometry of Cartan
fields. We begin by studying the contact analogous of hamiltonian fields of
symplectic geometry (in fact, they are the classical characteristic fields of
first order PDE's), along with several characterizations and properties. Then
(section \ref{subsec_Type_Cartan_Field}), we introduce the type of a Cartan
field $X$ as the rank of the system of Lie derivatives $U,X(U),X^{2}(U),...$.
Cartan fields of type $2$ or $3$ are characterized as linear combinations of
involutive hamiltonian fields or, equivalently, as those belonging to
integrable $2$-dimensional subdistributions of $\mathcal{C}$ (Theorem
\ref{Teoremone}); according with this property, normal forms are derived
(Theorems \ref{Tipo2_Tipo3}, \ref{Campi_Tipo2}).

In the final section, the main theorems are proved. In section
\ref{sec.inter.integrals} intermediate integrals, in the three senses
explained above (classical, nonholonomic, generalized), are considered. In
particular, we prove the aforementioned relations between existence of
intermediate integrals and integrability properties of $\mathcal{D}$ (Theorems
\ref{Goursat_Distributions}, \ref{Lychagin_Distributions} and Proposition
\ref{Equivalenza_Intermedio_Completo}). Using these results and normal forms
of Cartan fields, Theorem \ref{maintheor1} is proved (section
\ref{subsec_Maintheor1}). The following section contains the already mentioned
example of a wide class of parabolic MAE's admitting a complete integral,
together with an explicit computation of it. In Theorem
\ref{Better_Than_Bryant} we find the determining equation
(\ref{eq:equazione_impossibile}) for generalized intermediate integrals of
(\ref{MA_coordinate_bis}) and apply Cauchy-Kovalevsky theorem to prove the
existence of a solution in the analytic case: for what we said above, this is
equivalent to the existence of a complete integral of (\ref{MA_coordinate_bis}%
). Finally, results of section \ref{sec:geom.cartan.fields} on normal forms of
Cartan fields allow to obtain normal forms for degenerate lagrangian
distributions (Theorem \ref{Forme_Normali_Distribuzioni}), from which Theorem
\ref{maintheor2} immediately follows.

\paragraph{Notation and conventions.}
Throughout this paper, everything is supposed to be $C^{\infty}$
and local. For this reason, we do not lose in generality by
working with jets of sections rather than of submanifolds. For
simplicity, when $X$ is a vector field and $\mathcal{P}$ is a
distribution on the same manifold, we write
\textquotedblleft$X\in\mathcal{P}$" to mean that $X$ is a smooth
(local) section of tangent subbundle $\mathcal{P}$. We will use
$X(T)$ to denote the Lie derivative of a tensor $T$ along $X$.
Finally, first and
second order jet coordinates will be indifferently denoted with $z_{x}%
,z_{y},z_{xx},z_{xy},z_{yy}$ or $p,q,r,s,t$, respectively.

\section{Preliminary notions}

\label{sec.preliminary}

\subsection{Jet bundles and contact distribution}

Here we give the main definitions used in the present work. By $J^{r}(\tau)$
we denote the $r$-jet of the trivial bundle $\tau:\mathbb{R}^{2}%
\times\mathbb{R}\to\mathbb{R}^{2}$, i.e., the vector bundle of $r$-jets of
smooth functions on $\mathbb{R}^{2}$. These are equivalence classes of smooth
functions on $\mathbb{R}^{2}$ possessing the same partial derivatives up to
$r$-th order at a given point. Jet bundles of different orders are linked by
the obvious projections:
\[
\cdots\longrightarrow J^{2}(\tau)\overset{\tau_{2,1}}{\longrightarrow}%
J^{1}(\tau)\overset{\tau_{1,0}}{\longrightarrow}\mathbb{R}^{2}\times
\mathbb{R}\overset{\tau}{\longrightarrow}\mathbb{R}^{2}.
\]
For any $f\in C^{\infty}(\mathbb{R}^{2})$, let
$j_{r}f:\mathbb{R}^{2}\ni p\mapsto[f]^{r}_{p}\in J^{r}(\tau)$,
where $[f]^{r}_{p}$ is the $r$-jet of $f$ in $p$, be its $r$-th
order prolongation. \emph{R-planes} are the tangent planes to
graphs of $r$-th order prolongations. For any $[f]^r_p\in
J^{r}(\tau)$, $R$-planes passing through it biunivocally
correspond to $(r+1)$-jets projecting on $[f]^r_p$: namely,
$[f]^{r+1}_{p}$ corresponds to the tangent plane
$R_{[f]^{r+1}_{p}}$ to the graph of $j_{r}f$ at $[f]^{r}_{p}$
(\cite{Vinogradov and Co}).

\smallskip

A chart $(x,y,z)$ on the bundle $\tau$ induces a natural chart on each
$J^{r}(\tau)$. For instance $(x,y,z,p=z_{x},q=z_{y})$ are the induced
coordinates on $J^{1}(\tau)$ and $(x,y,z,p,q,r=z_{xx},s=z_{xy},t=z_{yy})$
those on $J^{2}(\tau)$. The $R$-plane $R_{\theta}\subset T_{\tau_{2,1}%
(\theta)}J^{1}(\tau)$ associated with point $\theta=(\bar{x},\bar{y},\bar
{z},\bar{p},\bar{q},\bar{r},\bar{s},\bar{t})\in J^{2}(\tau)$ is locally given
by
\[
R_{\theta}=<\left.  \widehat{\partial}_{x}\right|  _{\theta}+\bar{r}\left.
\partial_{p}\right|  _{\theta}+ \bar{s}\left.  \partial_{q}\right|  _{\theta
},\,\,\left.  \widehat{\partial}_{y}\right|  _{\theta}+\bar{s}\left.
\partial_{p}\right|  _{\theta}+\bar{t}\left.  \partial_{q}\right|  _{\theta
}>.
\]
where
\[
\widehat{\partial}_{x}\overset{\text{def}}{=}\partial_{x}+p\partial_{z},
\quad\widehat{\partial}_{y}\overset{\text{def}}{=}\partial_{y}+q\partial_{z}.
\]
The contact space $\mathcal{C}_{\theta}$ at $\theta\in J^{r}(\tau)$ is the
span of $R$-planes at $\theta$. The distribution $\theta\mapsto\mathcal{C}%
_{\theta}$ is then defined. From now on we shall focus on the case $r=1$. On
$J^{1}(\tau)$, $\mathcal{C}$ is given by:
\[
\mathcal{C}=<\widehat{\partial}_{x},\,\, \widehat{\partial}_{y},\,\,
\partial_{p},\,\,\partial_{q}>
\]
Dually, $\mathcal{C}$ is defined by $\{U=0\}$ where
\begin{equation}
\label{eq.U.vera.e.propria}U=dz-p\,dx-q\,dy.
\end{equation}
Note that, as what really matters is the distribution $\mathcal{C}$, one can
substitute $U$ with any other multiple of it. $(J^{1}(\tau),\mathcal{C})$ is a
contact manifold, i.e.
\begin{equation}
dU\wedge dU\wedge U\neq0 . \label{completa_non_integrablita}%
\end{equation}
Hence, by (\ref{eq.U.vera.e.propria}), $x,y,z,p,q$ are contact coordinates. As
is well known, (\ref{completa_non_integrablita}) is equivalent to the fact
that $\mathcal{C}$ does not admit integral submanifolds of dimension greater
than $2$, or also to the non existence of infinitesimal symmetries
$\mathcal{C}$ belonging to it. Furthermore, condition
(\ref{completa_non_integrablita}) is also equivalent to $(\mathcal{C}%
,dU|_{\mathcal{C}})$ being a symplectic vector bundle.

\medskip Note that, for any $X\in\mathcal{C}$, $X(U)=X\rfloor~dU$, i.e. one
can express orthogonality in $\mathcal{C}$ (with respect to $dU$) in terms of
Lie derivatives. For example, the orthogonal complement of $X$ in
$\mathcal{C}$ is described by
\[
X^{\bot}=\{U=0,~X(U)=0\}.
\]
In particular, $X^\bot$ is $3$-dimensional and contains $X$;
moreover, any $3$-dimensional subdistribution of $\mathcal{C}$ is
of this form. Analogously, if $\mathcal{D}\subset\mathcal{C}$ is a
distribution spanned by vector fields $X$, $Y$ then its orthogonal
complement is given by
\[
\mathcal{D}^{\bot}=\{U=0,X(U)=0,Y(U)=0\}.
\]
In particular, $\mathcal{D}$ is called a \emph{lagrangian distribution} if
$\mathcal{D}=\mathcal{D}^{\bot}$ (note that some other authors use the term
\emph{legendrian}).

\subsection{Parabolic Monge-Amp\`{e}re equations}

\label{sec.par.MAE}

Recall that a \emph{scalar differential equation} (in two independent
variables) $\mathcal{E}$ of order $r$ is a hypersurface of $J^{r}(\tau)$. A
\emph{solution} of $\mathcal{E}$ is a locally maximal integral manifold
$\Sigma$ of the restriction to $\mathcal{E}$ of the contact distribution on
$J^{r}(\tau)$; when $\Sigma$ is the graph of $j_{r}{f}$, with $f$ smooth
function on $\mathbb{R}^{2}$, then $\Sigma$ is a classical solution of
$\mathcal{E}$.

\smallskip Let $\mathcal{I}(U)\subset\Lambda^{*}(J^{1}(\tau))$ be the
differential ideal generated by $U$.

\begin{definition}
Let $\omega\in\Lambda^{2}\left(  J^{1}(\tau)\right)  \backslash\mathcal{I}%
(U)$. Let us associate with $\omega$ the scalar second order equation%
\begin{equation}
\mathcal{E}_{\omega}\overset{\text{def}}{=}\{\theta\in J^{2}(\tau)\text{ s.t.
}\omega|_{R_{\theta}}=0\}\text{,} \label{MA_Lychagin}%
\end{equation}
where $R_{\theta}\subset T_{\tau_{2,1}(\theta)}(J^{1}(\tau))$ is the $R$-plane
associated with $\theta$. The equations of this form are called
\emph{Monge-Amp\`{e}re equations} (see \cite{KLR,Lychagin contact}).
\end{definition}

In other words $\mathcal{E}_{\omega}$ is the differential equation
corresponding to the exterior differential system $\{U=0,\,\omega=0\}$.

\medskip\noindent\textbf{Coordinate expression}. Denote by $(x,y,z,p,q,r,s,t)$
a system of local contact coordinates on $J^{2}(\tau)$. In such a chart, a
generic MAE takes the form
\begin{equation}
N(rt-s^{2})+Ar+Bs+Ct+D=0\text{,} \label{MA_coordinate}%
\end{equation}
with $N,A,B,C,D\in C^{\infty}(J^{1}(\tau))$. The $2$-forms $\omega$ on
$J^{1}(\tau)$ such that $\mathcal{E}_{\omega}$ is given by
(\ref{MA_coordinate}) are
\begin{align*}
\omega=D~dx\wedge dy +\left(  \frac{B}{2}+b\right)  dx\wedge dp &
+~C~dx\wedge dq-A~dy\wedge dp\label{MA_simplettica}\\
&  +\left(  -\frac{B}{2}+b\right)  ~dy\wedge dq+N~dp\wedge
dq+\alpha\wedge
U\text{,}%
\end{align*}
with arbitrary $b\in C^{\infty}(J^{1}(\tau))$, $\alpha\in\Lambda^{1}%
(J^{1}(\tau))$.

\medskip It is clear from the above formula that the correspondence
$\omega\longmapsto\mathcal{E}_{\omega}$ is not invertible. Let us consider in
$\Lambda^{2}(J^{1}(\tau))\backslash\mathcal{I}(U)$ the following equivalence
relation: \
\begin{equation}
\omega\sim\rho\quad\Longleftrightarrow\quad\exists\,\mu\neq
0,\lambda\in C^{\infty
}(J^{1}(\tau))\text{ s.t. }\rho|_{\mathcal{C}}=\mu\omega|_{\mathcal{C}%
}+\lambda(dU)|_{\mathcal{C}} \label{forme_equivalenti}%
\end{equation}
(or $\rho=\mu\omega+\lambda dU+\alpha\wedge U$ for some $1$-form $\alpha$).

\medskip It can be proved (see \cite{KLR}) that two $2$-forms on $J^{1}(\tau)$
are equivalent in the sense of (\ref{forme_equivalenti}) if and only if they
define the same MAE.

\begin{proposition}
For any $\omega\in\Lambda^{2}(J^{1}(\tau))\backslash\mathcal{I}(U)$, there are
at most two $2$-forms equivalent, up to a factor, to $\omega$ in the sense of
(\ref{forme_equivalenti}) and such that their restriction to $\mathcal{C}$ is
degenerate (so that they are decomposable).
\end{proposition}

\begin{proof}
The restriction to $\mathcal{C}$ of a $2$-form equivalent to $\omega$ is, up
to a factor, always of the form $\omega_{\lambda}=(\omega+\lambda
dU)|_{\mathcal{C}}$ with $\lambda\in C^{\infty}(J^{1}(\tau))$. On the other
hand, it is easy to see that $Rad~\omega_{\lambda}$ is non trivial if and only
if $\omega_{\lambda}\wedge\omega_{\lambda}=0$, i.e.
\begin{equation}
(\omega\wedge\omega+2\lambda\omega\wedge dU+\lambda^{2}dU\wedge
dU)|_{\mathcal{C}}=(\alpha+2k\lambda+\lambda^{2})(dU\wedge dU)|_{\mathcal{C}%
}=0\text{,} \label{equazione_lambda}%
\end{equation}
with
\[
(\omega\wedge\omega)|_{\mathcal{C}}=\alpha(dU\wedge dU)|_{\mathcal{C}}\text{
\thinspace\thinspace\ and \thinspace\thinspace\ }(\omega\wedge
dU)|_{\mathcal{C}}=k(dU\wedge dU)|_{\mathcal{C}}.
\]
As (\ref{equazione_lambda}) is quadratic in $\lambda$, the proposition is proved.
\end{proof}

\smallskip Note that the sign of the discriminant $k^{2}-\alpha$ in
(\ref{equazione_lambda}) is the same of the expression%
\[
\Delta=B^{2}-4AC+4ND.
\]

Let us recall the following basic notion.

\begin{definition}
Let $\mathcal{E}\subset J^{2}(\tau)$ be a second order scalar differential
equation, and let $\theta\in J^{1}(\tau)$. A line $r\subset\mathcal{C}%
_{\theta}$ is said to be \emph{characteristic} for $\mathcal{E}$ in $\theta$
if it belongs to more than one $R$-plane $R_{\widehat{\theta}}$, with
$\widehat{\theta}\in\mathcal{E}\cap\tau_{2,1}^{-1}(\theta)$.
\end{definition}

Characteristic directions are those belonging to more than one integral
manifold of $\mathcal{E}$: a curve $\gamma\subset J^{1}(\tau)$ (locally)
determines the integral surface passing through it if and only if the tangent
lines to $\gamma$ are not characteristic. In the case of MAE's, it is not
difficult to check that a line is characteristic for $\mathcal{E}_{\omega}$ if
and only if it belongs to the radical of a degenerate $2$-form $\omega
_{\lambda}$ equivalent to $\omega$. According to the previous proposition,
there are three possibilities:

\begin{itemize}
\item[1)] if $\Delta>0$, there are two distinct $\lambda$'s such that
$Rad~\omega_{\lambda}\neq0$; hence, there exist two distinct families of
characteristic lines (\emph{hyperbolic case});

\item[2)] if $\Delta=0$, there is just one $\lambda$ for which $\omega
_{\lambda}$ is degenerate; in this case there is only one family of
characteristics (\emph{parabolic case});

\item[3)] if $\Delta<0$, $\omega_{\lambda}$ is always non degenerate, so that
there are no characteristics (\emph{elliptic case}).
\end{itemize}

\noindent\emph{\textbf{Warning.}} As the paper is devoted to the parabolic case, from now on, when writing $\mathcal{E}%
_{\omega}$, we mean that $\omega$ is (up to a factor) the only
degenerate representative of the equation.

\begin{definition}
The $2$-dimensional distribution $\mathcal{D}=Rad~\omega|_{\mathcal{C}}$ is
called \emph{characteristic distribution} of the parabolic MAE $\mathcal{E}%
_{\omega}$.
\end{definition}

\begin{proposition}
Let $\mathcal{E}_{\omega}$ be a parabolic MAE. Then its characteristic
distribution $\mathcal{D}$ is lagrangian. Conversely, any lagrangian
distribution is characteristic for one and only one parabolic MAE.
\end{proposition}

\begin{proof}
We must prove that $dU|_{\mathcal{D}}=0$, which is equivalent to $\omega\wedge
dU|_{\mathcal{C}}=0$. But, by equation (\ref{equazione_lambda}) and the
assumptions made, one has $\alpha=k=0$, and the proposition follows.
\end{proof}

\smallskip Note that, by the above proposition, a parabolic MAE can be
specified by assigning a lagrangian subdistribution of $\mathcal{C}$. In fact,
let $\mathcal{D}=<X,Y>$. Then the corresponding MAE is $\mathcal{E}_{\omega}$,
with
\[
\omega=X(U)\wedge Y(U).
\]
If the generators of $\mathcal{D}$ are locally expressed by%
\begin{equation}
X=\widehat{\partial}_{x}+R\partial_{p}+S\partial_{q}\,,\quad Y=\widehat
{\partial}_{y}+S\partial_{p}+T\partial_{q}\text{,} \label{generators_D}%
\end{equation}
with $R$, $S$, $T\in C^{\infty}(J^{1}(\tau))$, then%
\[
X(U)=dp-Rdx-Sdy\,,\quad Y(U)=dq-Sdx-Tdy\,,
\]
from which it follows that equation $\mathcal{E}_{X(U)\wedge Y(U)}\subset
J^{2}(\tau)$ is%
\[
(s-S)^{2}-(r-R)(t-T)=0.
\]
Recall that the Legendre transformation maps $\widehat{\partial}_{x}%
,\widehat{\partial}_{y},\partial_{p},\partial_{q}$ into $\partial_{p}%
,\partial_{q},-\widehat{\partial}_{x},-\widehat{\partial}_{y}$, respectively.
A partial Legendre transformation just exchanges $\widehat{\partial}_{x}$ with
$\partial_{p}$ or $\widehat{\partial}_{y}$ with $\partial_{q}$ (up to a sign).
Thus expression (\ref{generators_D}) is the most general coordinate
representation of $\mathcal{D}$, up to contact transformations.

\section{Geometry of Cartan fields}

\label{sec:geom.cartan.fields} As the contact distribution $\mathcal{C}$ is
completely non integrable, the flow of any Cartan field $X\in\mathcal{C}$
deforms it; the sequence of iterated Lie derivatives%
\begin{equation}
U,X(U),X^{2}(U),X^{3}(U) \label{X_i(U)}%
\end{equation}
gives a measure of this deformation (as $J^{1}(\tau)$ is $5$-dimensional and
all the forms $X^{j}(U)$ vanish on $X$, there is no need to consider the
remaining derivatives).

\begin{definition}
Let $X\in\mathcal{C}$. The \emph{type} of $X$ is the rank of system
(\ref{X_i(U)}).
\end{definition}

The following cases are possible:

\begin{itemize}
\item[1)] Fields of \emph{type 2}: $X^{2}(U)$ depends on $U$ and $X(U)$ (which
is equivalent to $X$ being characteristic for $X^{\perp}=\{U=X(U)=0\}$);

\item[2)] Fields of \emph{type 3}: $U$, $X(U)$, $X^{2}(U)$ are independent but
$X^{3}(U)$ depends on them (which is equivalent to\textbf{ }$X$ being
characteristic for distribution $\{U=X(U)=X^{2}(U)=0\}$);

\item[3)] Fields of \emph{type 4}: $U$, $X(U)$, $X^{2}(U)$, $X^{3}(U)$ are independent.
\end{itemize}

Note that, due to the complete non integrability of the contact distribution,
it can not be $X(U)=\lambda U$, for $X\in\mathcal{C}\backslash\{0\}$ (``type
$1$"). Note also that the above three cases are well defined, i.e. they do not
depend on the choice of $U$ nor on the length of $X$ (in other words, what we
are dealing with are line distributions, rather than vector fields). As one
can realize from the definition, the higher is the type, the more complicated
is the structure of Cartan fields.

\smallskip In the rest of the section we will study the main properties of
different types of Cartan fields, starting from the simplest and the most
basic ones: hamiltonian vector fields.

\subsection{Hamiltonian fields and integrable distributions}

The map%
\begin{equation}
\chi:\mathcal{C}\longrightarrow\Lambda^{1}(J^{1}(\tau
))/<U>\,,\,\,\,X\longmapsto X(U)\,\,\text{mod}<U>\text{,} \label{Chi_map}%
\end{equation}
is a $C^{\infty}(J^{1}(\tau))$-module isomorphism: it associates
with each Cartan field $X$ the restriction of $X(U)$ to
$\mathcal{C}$. Note that, although $\chi$ depends on the choice of
$U$ (by substituting it with a multiple $\overline{U}=\lambda U$
one gets $\overline{\chi}=\lambda\chi$), $Rad~\chi(X)=X^{\bot}$
does not change. By inverting $\chi$, with each
$\sigma\in\Lambda^{1}(J^{1}(\tau))$ one associates a Cartan vector
field
\[
X_{\sigma}\overset{\text{def}}{=}\chi^{-1}([\sigma])\text{,}%
\]
where $[\sigma]$ is the equivalence class of $\sigma$ in $\Lambda^{1}%
(J^{1}(\tau))/<U>$; in other words, $X_{\sigma}\in\mathcal{C}$ is determined
by the relation%
\begin{equation}
X_{\sigma}(U)=X_{\sigma}\lrcorner dU=\sigma+\lambda U \label{X_sigma}%
\end{equation}
for some $\lambda\in C^{\infty}(J^{1}(\tau))$ (in fact, if $U$ is given by
(\ref{eq.U.vera.e.propria}), then $\lambda=-\sigma(\partial_{z}))$.

\begin{proposition}
\label{three dimensions}$X_{\sigma}^{\bot}=\{U=0,\sigma=0\}$. Furthermore,
$X_{\sigma}$ is characteristic for $X_{\sigma}^{\bot}$ if and only if it is of
type $2$.
\end{proposition}

\begin{proof}
It follows from (\ref{X_sigma}) that $\sigma(X_{\sigma})=0$. But then%
\[
X_{\sigma}(\sigma)=X_{\sigma}^{2}(U)-X_{\sigma}(\lambda)U-\lambda X_{\sigma
}(U),
\]
hence, $X_{\sigma}$ is characteristic for $X^{\bot}_{\sigma}$ if and only if
$X_{\sigma}^{2}(U)$ linearly depends on $U$ and $X_{\sigma}(U)$.
\end{proof}

\medskip In the case $\sigma$ is exact, $\sigma=df$, we simply write $X_{f}$
instead of $X_{df}$. Due to the apparent analogy with the case of symplectic
geometry, we give the following

\begin{definition}
\label{Hamiltonian_Field}Let $f\in C^{\infty}(J^{1}(\tau))$, then the vector
field $X_{f}\in\mathcal{C}$ is called the (contact-)hamiltonian vector field
associated with $f$.
\end{definition}

Note that, although $X_{f}$ depends on the particular choice of $U$, its
direction, which by the previous proposition is orthogonal to $\{U=0,df=0\}$,
only depends on $\mathcal{C}$ (and $f\,$, of course). Furthermore, as in the
symplectic case, $f$ is a first integral of the corresponding field:
$X_{f}(f)=df(X_{f})=0$, from which it easily follows that $X_{f}$ is of type
$2$. By the previous proposition, $X_{f}$ is characteristic for distribution
$X_{f}^{\bot}$: in other words, $X_{f}$ coincides with the classical
characteristic vector field of the first order equation $f=0$. Its local
expression in a contact coordinate system $(x,y,z,p,q)$ on $J^{1}(\tau)$ is
\[
X_{f}=\partial_{p}(f)\,\widehat{\partial}_{x}+\partial_{q}(f)\,\widehat
{\partial}_{y}-\widehat{\partial}_{x}(f)\,\partial_{p}-\widehat{\partial}%
_{y}(f)\,\partial_{q}.
\]
In particular:
\[
X_{x}=-\partial_{p},\quad X_{y}=-\partial_{q},\quad X_{z}=-p\partial
_{p}-q\partial_{q},\quad X_{p}=\widehat{\partial}_{x},\quad X_{q}%
=\widehat{\partial}_{y}.
\]

\begin{example}
Let $X\in\mathcal{C}$, and $f$ be a first integral of $X$ then $X_{f}\in
X^{\bot}$:%
\[
dU(X_{f},X)=X_{f}(U)(X)=(df+\lambda U)(X)=0.
\]
Hence, if $f$, $g$, $h$ are three first integrals such that $df$, $dg$, $dh$,
$U$ are independent, then%
\[
X^{\bot}=<X_{f},X_{g},X_{h}>.
\]

\end{example}

\begin{theorem}
\label{th.sostiutivo} Let $f,g\in C^{\infty}(J^{1}(\tau))$. Then the following
properties are equivalent:

\begin{itemize}
\item[1)] the distribution $<X_{f},X_{g}>$ is integrable$;$

\item[2)] $X_{f}$ and $X_{g}$ are orthogonal with respect to $dU$;

\item[3)] $X_{f}(g)=X_{g}(f)=0$;
\end{itemize}

Furthermore, if $f,g$ are functionally independent, then the following two
properties can be added to the above list of equivalences:

\begin{itemize}
\item[4)] there exists a third function $h\in C^{\infty}(J^{1}(\tau))$ such
that $U$ linearly depends on $df,dg,dh$;

\item[5)] there exists a system of contact coordinates $(x,y,z,p,q)$ in which
$x=f$, $y=g$;
\end{itemize}
\end{theorem}

\begin{proof}
$1)$ implies $2)$. It follows from%
\[
dU(X_{f},X_{g})=-U([X_{f},X_{g}])
\]
and from the fact that $[X_{f},X_{g}]$ depends on $X_{f}$ and $X_{g}$. Also,
$2)$ implies $1)$. It follows from%
\[
dU(X_{f},[X_{f},X_{g}])=df([X_{f},X_{g}])=X_{f}\left(  X_{g}(f)\right)
-X_{g}\left(  X_{f}(f)\right)  =0
\]
and the analogous relation for $X_{g}$, keeping in mind that $<X_{f}%
,X_{g}>^{\bot}=<X_{f},X_{g}>$.

\smallskip The equivalence of $2)$ and $3)$ is an immediate consequence of
(\ref{X_sigma}) applied to the cases $\sigma=df$ and $\sigma=dg$, respectively.

Let us now assume the functional independence of $f$ and $g$. If $1)$ holds,
then by $2)$ $df$ and $dg$ vanish on $<X_{f},X_{g}>$, so that there exists a
third function $h$, independent from $f$ and $g$, such that $<X_{f}%
,X_{g}>=\{df=0,dg=0,dh=0\}$. As $U$ vanishes on $X_{f}$, $X_{g}$ it linearly
depends on $df,dg,dh$.

\smallskip Let now $4)$ hold, then:%
\begin{equation}
\lambda U=dh-adf-bdg\text{,} \label{Tre_Tenori}%
\end{equation}
for some functions $\lambda,a,b\in C^{\infty}(J^{1}(\tau))$ (note that, as $U$
is completely non integrable, in (\ref{Tre_Tenori}) all the three
differentials must appear). But, then%
\[
x=f,\,\,y=g,\,\,h=z,\,\,p=a,\,\,q=b
\]
are contact coordinates on $J^{1}(\tau)$, which proves $5)$. Finally, let $5)$
hold. Then%
\[
X_{f}(g)=X_{x}(y)=-\partial_{p}(y)=0\text{,}%
\]
which implies $3)$.
\end{proof}

\smallskip We note that the previous theorem is a special case of a more
general result, essentially due to Jacobi (the statement and proof can be
found in \cite{Muñoz Diaz}).


\begin{definition}
Two functions $f,g\in C^{\infty}(J^{1}(\tau))$ are \emph{in involution }when
they satisfy any of the equivalent properties $1)$, $2)$, $3)$ of the previous theorem.
\end{definition}

\begin{theorem}
[structure of integrable distributions]\label{Caratterizzazione_Integrabilita}
Let $\mathcal{D}$ be a 2-dimensional distribution in $\mathcal{C}$. Then
$\mathcal{D}$ is integrable if and only if it is spanned by two hamiltonian
fields $X_{f}$ e $X_{g}$, with $f$ and $g$ independent and in involution.
\end{theorem}

\begin{proof}
One of the two implications has already been proved in the previous theorem.
As to the converse implication, let $\mathcal{D}\subset\mathcal{C}$ be
$2$-dimensional and integrable. Then $\mathcal{D}=\{df=dg=dh=0\}$ for some
independent functions $f$, $g$, $h$. But, as $U$ vanishes on $\mathcal{D}$, it
linearly depends on $df$, $dg$, $dh$, i.e. is of the form (\ref{Tre_Tenori}),
so that, by the same argument used there, $\ f$ e $g$ are in involution (and ,
obviously, $\mathcal{D}$ contains $X_{f}$ and $X_{g}$).
\end{proof}

\medskip As a consequence of Theorems \ref{Caratterizzazione_Integrabilita}
and $5)$ of Theorem \ref{th.sostiutivo} one has that every integrable,
2-dimensional distribution in $\mathcal{C}$ can be reduced to the form%
\[
\mathcal{D}=<\partial_{p},\partial_{q}>=<X_{x},X_{y}>
\]
in a suitable contact chart; a partial or total Legendre map gives the
alternative representations%
\[
\mathcal{D}=<\widehat{\partial}_{x},\widehat{\partial}_{y}>\,\,\,\text{or}%
\,\,\,\,\mathcal{D}=<\widehat{\partial}_{x},\partial_{q}>\,\,\,\text{or}%
\,\,\,\,\mathcal{D}=<\widehat{\partial}_{y},\partial_{p}>.
\]
The following proposition, together with Proposition \ref{three dimensions},
completes the discussion of integrability of subdistributions in $\mathcal{C}$.

\begin{proposition}
\label{3_mai_integrabile}Let $\mathcal{P}\subset\mathcal{C}$ be a
$3$-dimensional distribution. Then its derived distribution $\mathcal{P}%
^{\prime}$ is not contained in $\mathcal{C}$; in particular, $\mathcal{P}$ is
not integrable.
\end{proposition}

\begin{proof}
Assume, by contradiction, that $\mathcal{P}^{\prime}\subset\mathcal{C}$. Then,
for any couple of fields $X_{1}$, $X_{2}\in\mathcal{P}$ it would hold
$dU(X_{1},X_{2})=-U([X_{1},X_{2}])=0$, i.e. $(dU)|_{\mathcal{C}}$ would
identically vanish on $\mathcal{P}$.
\end{proof}


\medskip Below we will need the following general lemma on derived
distributions. The proof is straightforward.

\begin{lemma}
\label{Lemma_Calcolo_Derivato} Let $\mathcal{P}$ be a $k$-dimensional
distribution on a smooth manifold $M^{n}$ and let $I_{\mathcal{P}}$ be the
corresponding Pfaffian system. Then the Pfaffian system associated with the
derived distribution $\mathcal{P}^{\prime}$ is:%
\[
I_{\mathcal{P}}^{\prime}=\{\omega\in I_{\mathcal{P}}\text{ s.t. }X(\omega)\in
I_{\mathcal{P}} \,\, \forall X\in\mathcal{P}\}.
\]

\end{lemma}

The next proposition characterizes hamiltonian fields by integrability
properties of their orthogonal complements.

\begin{proposition}
\label{caso_4_4}Let $X\in\mathcal{C}$. Then $X$ is a multiple of a hamiltonian
field $X_{f}$ if and only if $(X^{\perp})^{\prime}$ is $4$-dimensional and integrable.
\end{proposition}

\begin{proof}
Assume $X=X_{f}$ (or a multiple of it), then%
\[
X^{\perp}=X_{f}^{\perp}=<X_{f},X_{g},X_{h}>
\]
with $g$ e $h$ being independent first integrals of $X$ obviously in
involution with $f$. On the other hand
\[
X_{f}^{\perp}=\{U=X_{f}(U)=0\}=\{U=df=0\}.
\]
Furthermore%
\[
X_{f}(df)=X_{g}(df)=X_{h}(df)=0
\]
i.e., by the previous lemma, $df$ belongs to the derived system of $<U,df>$.
Hence,%
\[
(X_{f}^{\perp})^{\prime}=\{df=0\}
\]
which is $4$-dimensional and integrable.

\smallskip Viceversa, let ($X^{\perp})^{\prime}$ be $4$-dimensional and
integrable, then there exists a function $f$ such that ($X^{\perp})^{\prime
}=\{df=0\}$; therefore
\[
X^{\perp}=(X^{\perp})^{\prime}\cap\mathcal{C}=\{U=df=0\}=\{U=X_{f}%
(U)=0\}=X_{f}^{\perp}%
\]
which entails the parallelism between $X$ and $X_{f}$.
\end{proof}

\subsection{Cartan fields of type 2}

\label{subsec_Type_Cartan_Field}


The following result generalizes Proposition \ref{caso_4_4} and gives a
characterization of type $2$ Cartan fields.

\begin{proposition}
\label{Ortogonale_Tipo_2} Let $X\in\mathcal{C}$. Then $X$ is of type $2$ if
and only if the derived distribution $(X^{\bot})^{\prime}$ has dimension $4$.
\end{proposition}

\begin{proof}
Let $\dim(X^{\bot})^{\prime}=4$. Then, by Lemma \ref{Lemma_Calcolo_Derivato}
applied to the case $\mathcal{P}=X^{\bot}$, $(X^{\bot})^{\prime}$ is described
by equation $\sigma=0$, with $\sigma$ linear combination of $U$ and $X(U)$%
\begin{equation}
\sigma=\alpha U+X(U) \label{forma_sigma}%
\end{equation}
(by Proposition \ref{3_mai_integrabile}, $\sigma$ is not a multiple of $U$)
and such that, for any $W\in X^{\bot}$, $W(\sigma)$ linearly depends on $U$
and $X(U)$. In particular,
\[
X^{2}(U)\equiv X(\sigma)\equiv0\,\,\text{mod}<U,X(U)>.
\]
Viceversa, let $X$ be of type $2$. To prove our statement we must find an
$\alpha$ in (\ref{forma_sigma}) such that $X^{\bot}$ is described by equation
$\sigma=0$. To this end, let $\{X,Y,Z\}$ be a basis of $X^{\bot}$, then
$X(\sigma)$, $Y(\sigma)$ and $Z(\sigma)$ must vanish on $X^{\bot}$. By
assumption it holds%
\[
X(\sigma)=X^{2}(U)+X(\alpha)U+\alpha X(U)\equiv0\,\,\text{mod}<U,X(U)>
\]
and, therefore, $X(\sigma)$ vanishes on $X^{\bot}$ for any choice of $\alpha$.
As to $Y(\sigma)$, relations%
\[
Y(\sigma)(X)=-X(\sigma)(Y)=0\,,\quad Y(\sigma)(Y)=d\sigma(Y,Y)=0
\]
hold true for any $\alpha$, whereas equation%
\[
0=Y(\sigma)(Z)=dX(U)(Y,Z)+\alpha\, dU(Y,Z)
\]
determines $\alpha$.
Therefore, by choosing $\alpha$ in this way, one has that $Y(\sigma)$ vanishes
on $X^{\bot}$; the same holds for $Z(\sigma)$, due to the symmetry of roles of
$Y$ and $Z$.
\end{proof}

\begin{proposition}
\label{Unico_Tipo2}Let $\mathcal{D}\subset\mathcal{C}$ be a lagrangian, non
integrable distribution. Then, it contains at most one field of type $2$; if
such a field exists, it spans $(\mathcal{D}^{\prime})^{\bot}$.
\end{proposition}

\begin{proof}
Let $X\in\mathcal{D}$ be of type $2$. Then, if $\mathcal{D}=<X,Y>$, it holds
$dU(X,X)=dU(X,Y)=0$ and
\[
dU(X,[X,Y])=X(U)([X,Y])=X(X(U)(Y))-X^{2}(U)(Y)=0.
\]

\end{proof}



\begin{proposition}
\label{Coppia_Per_Tipo2}Let $X\in\mathcal{C}$ be of type $2$. For any first
integral $f$ of $X$ the distribution $<X,X_{f}>$ is integrable. Conversely,
every 2-dimensional integrable distribution in $\mathcal{C}$ which contains
$X$ is of this form.
\end{proposition}

\begin{proof}
Let $f\in C^{\infty}(J^{1}(\tau))$ be a first integral of $X$, then the
lagrangian distribution $\mathcal{D}=<X,X_{f}>$ is integrable. In fact,
$[X,X_{f}]\in\mathcal{D}$ if and only if it is orthogonal to both $X$ and
$X_{f}$. But%
\[
dU(X_{f},[X,X_{f}])=(df-f_{z}U)([X,X_{f}])=df([X,X_{f}])=X(X_{f}%
(f))-X_{f}(X(f))=0
\]
(this holds for any $X\in\mathcal{C}$ having $f$ as a first integral) and also%
\[
dU(X,[X,X_{f}])=X(dU(X,X_{f}))-dX(U)(X,X_{f})=0-X^{2}(U)(X_{f})=0
\]
because $X^{2}(U)$ depends on $U$ and $X(U)$.

\smallskip Viceversa, let $\mathcal{D}\subset\mathcal{C}$ be a 2-dimensional
integrable distribution. Then, by Theorem
\ref{Caratterizzazione_Integrabilita}, $\mathcal{D}=<X_{f},X_{g}>$ with $f$
and $g$ in involution. Therefore, if $X\in\mathcal{D}$, then $f$ and $g$ are
first integrals of $X$.
\end{proof}


\subsection{Normal forms of Cartan fields}

In this section normal forms for Cartan fields are given. The following
proposition gives us the simplest possible form valid for any Cartan field.
For fields of type less than $4$, more precise normal forms can be obtained.
These are a consequence of next theorem, which characterizes non generic
Cartan fields in terms of involutive hamiltonian fields.

\begin{proposition}
For any field $X\in\mathcal{C}$ there exists a contact coordinate system in
which $X$ takes the form%
\[
X=a\widehat{\partial}_{x}+b\partial_{p}+c\partial_{q}\,, \quad a,b,c\in
C^{\infty}(J^{1}(\tau)).
\]

\end{proposition}

\begin{proof}
Let $f$ be a first integral of $X$ (equivalently, $X_{f}$ be orthogonal to
$X$), then one may assume, according to Theorem \ref{th.sostiutivo}, that in a
certain contact chart is $f=y$ and consequently $X_{f}=\partial_{q}$, from
which the statement follows, because $\partial_{q}^{\perp}$ is spanned by
$\widehat{\partial}_{x}$, $\partial_{p}$, $\partial_{q}$.
\end{proof}

\begin{theorem}
\label{Teoremone} Let $X\in\mathcal{C}$, then the following equivalences hold:

\begin{itemize}
\item[1)] $X$ is of type $2$ or $3$;

\item[2)] $X=a X_{f}+b X_{g}$ with $f$ and $g$ in involution and
$a,b\in C^{\infty}(J^{1}(\tau))$;

\item[3)] $X=a\partial_{p}+b\partial_{q}$ in an appropriate
contact chart $(x,y,z,p,q)$, and $a,b\in C^{\infty}(J^{1}(\tau))$;

\item[4)] $X$ admits two independent first integrals in involution;

\item[5)] $X$ belongs to at least one $2$-dimensional integrable
subdistribution of $\mathcal{C}$.
\end{itemize}
\end{theorem}

\begin{proof}
1) implies 2). In fact, if $X$ is of type $2$ then the statement follows from
Proposition \ref{Coppia_Per_Tipo2}. If, instead, $X$ is of type $3$, then it
is characteristic for the distribution $\mathcal{D}_{X}=\{U=X(U)=X^{2}%
(U)=0\}=<X,Y>$, for some $Y\in X^{\perp}$. Hence, $\mathcal{D}_{X}$ is
integrable (because it contains $[X,Y]$) and, consequently, it is spanned by
two vector fields in involution (Theorem \ref{Caratterizzazione_Integrabilita}%
). Also, 2) implies 1). In fact, if we put $X^{0}=\text{id}$, in this case the
following relations hold:
\[
X^{j}(U)\equiv X^{j-1}(a)df+X^{j-1}(b)dg \,\,\,\text{mod}<U,\dots
,X^{j-1}(U)>, \quad 1\leq j \leq3.
\]
from which the linear dependence of $U,X(U),X^{2}(U),X^{3}(U)$ follows.

\smallskip Equivalence between 2) and 3) immediately follows from $4)$ of
Theorem \ref{th.sostiutivo}. Equivalence between 2) and 5) is just Theorem
\ref{Caratterizzazione_Integrabilita}.

\smallskip4) trivially follows from 2). Now, assuming 4) to hold, let $f$ and
$g$ be the two (independent) involutive first integrals, then: $X(f)=X(g)=0$,
$X_{f}(g)=0$, or also, in terms of orthogonality, $X\in<X_{f},X_{g}>^{\perp
}=<X_{f},X_{g}>$.
\end{proof}

\begin{remark}
We have already proved (Proposition \ref{Coppia_Per_Tipo2}) that, if
$X\in\mathcal{C}$ is of type $2$, then it is contained in a family of
$2$-dimensional integrable subdistributions of $\mathcal{C}$ (one for each
first integral). On the other hand, if $X$ is of type $3$, it is contained in
just one $2$-dimensional integrable subdistribution of $\mathcal{C}$, namely
the distribution $\mathcal{D}_{X}$ defined in the proof of the above theorem.
\end{remark}


We have seen in Theorem \ref{Teoremone} that, modulo a contact
transformation, a field $X\in\mathcal{C}$ of type less than $4$
takes the form $X=\partial _{p}+b\partial_{q}$ (as the type of a
field depends only on its direction, we have chosen $a=1$ in point
3) of above theorem). Then $X(U)=-dx-b dy$ and $X^{2}(U)=-X(b)dy$
from which it follows that $X^{2}(U)$ depends on $U$ and $X(U)$ if
and only if $b$ is a first integral of $X$. Therefore, on gets the
following

\begin{theorem}
\label{Tipo2_Tipo3} Let $X\in\mathcal{C}$. Then

\begin{enumerate}
\item[1)] $X$ is of type $2$ if and only if, in a suitable contact chart, it
takes the form
\begin{equation}
X=a\partial_{p}+b\partial_{q}\,,\,\,\, \text{with}\,\, X(b/a)=0;
\label{Tipo_2}%
\end{equation}

\item[2)] $X$ is of type $3$ if and only if (\ref{Tipo_2}) holds,
with $X(b/a)\neq0$.
\end{enumerate}
\end{theorem}

This result can be refined in the case of a field of type $2$.

\begin{theorem}
\label{Campi_Tipo2}A vector field $X\in\mathcal{C}$ is of type $2$
if and only if, in some contact chart, it takes one of the forms
\[
X=\partial_{p} \,\,\,\,\text{or}\,\,\,\, X=\partial_{p}+z \partial_{q}.
\]

\end{theorem}

\begin{proof}
Let $(X^{\bot})^{\prime}$ be locally described by equation $\sigma=0$ (see
also Proposition \ref{Ortogonale_Tipo_2}). By Darboux theorem, one can choose
independent functions $f,g,h,k,l$ in such a way that, up to a factor, one of
the following three expressions holds: either%
\begin{equation}
\sigma=df \label{(4,5)_zera}%
\end{equation}
or%
\begin{equation}
\sigma=df-gdh \label{(4,5)_prima}%
\end{equation}
or
\begin{equation}
\sigma=df-gdh-kdl. \label{(4,5)_seconda}%
\end{equation}
Expression (\ref{(4,5)_seconda}) can be excluded because, otherwise,
$\{\sigma=0\}$ would be a contact structure containing a $3$-dimensional
distribution, $X^{\bot}$, such that $(X^{\bot})^{\prime}=\{\sigma=0\}$, which
is impossible by Proposition $\ref{3_mai_integrabile}$. If (\ref{(4,5)_zera})
holds, $X$ is a multiple of $X_{f}$ (Proposition \ref{caso_4_4}); on the other
hand, by Theorem \ref{th.sostiutivo}, there exists a contact transformation
sending $f$ into coordinate $x$, so that, modulo a factor,%
\[
X=\partial_{p}.
\]
Finally, in case (\ref{(4,5)_prima}) one has%
\begin{equation}
\label{X_di_appoggio}X=X_{\sigma}=X_{f}-gX_{h}.
\end{equation}
Hence,%
\[
X(U)=df-gdh-(f_{z}-gh_{z})U\,, \quad X^{2}(U)=-X_{h}(f)dg+X_{g}(f-gh)dh.
\]
But, being $X$ of type $2$, one gets%
\begin{equation}
-X_{h}(f)dg+X_{g}(f-gh)dh=\lambda U+\mu(df-gdh) \label{bingo3}%
\end{equation}
for some $\lambda,\mu\in C^{\infty}(J^{1}(\tau))$. As the contact form $U$ is
determined up to a factor, one may assume that $\lambda$ does not vanish.
Hence, it follows from (\ref{bingo3}) that%
\[
U=-\frac{X_{h}(f)}{\lambda} \left(  dg +\frac{\mu}{X_{h}(f)}df + \frac
{X_{g}(gh-f)-\mu g}{X_{h}(f)}dh \right)  .
\]
Hence the functions
\[
x=f,\quad y=h,\quad z=-g,\quad p = \frac{\mu}{X_{h}(f)},\quad q=\frac
{X_{g}(gh-f)-\mu g}{X_{h}(f)}%
\]
form a contact chart. Consequently, $X$ of (\ref{X_di_appoggio}) assumes the
form
\[
X=X_{x}+zX_{y}=\partial_{p}+z\partial_{q}.
\]

\end{proof}

\smallskip As a remarkable application of normal form (\ref{Tipo_2}), we prove
the following proposition.

\begin{proposition}
Let $\mathcal{D}\subset\mathcal{C}$ be a non integrable lagrangian
distribution, and let $(\mathcal{D}^{\prime})^{\bot}$ be spanned by vector
field $X$. Then $X$ is not of type $3$.
\end{proposition}

\begin{proof}
Assume the type of $X$ less than $4$. Then it is $2$ or $3$. By Theorem
\ref{Tipo2_Tipo3}, in some contact coordinates $X$ takes the form
\[
X=\partial_{p}+a\partial_{q}\,,\,\,\,a\in C^{\infty}(J^{1}(\tau))
\]
(as the type only depends on the direction of $X$, the coefficient of
$\partial_{p}$ in (\ref{Tipo_2}) can be assumed equal to $1$). Let
$\mathcal{D}=<X,Y>$, then $Y\in X^{\bot}$ and, hence, is of the form%
\[
Y=\widehat{\partial}_{x}-\frac{1}{a}\widehat{\partial}_{y}+b\partial
_{p}+c\partial_{q}\text{,}%
\]
for some functions $b,c\in C^{\infty}(J^{1}(\tau))$. Let us now impose the
orthogonality between $X$ and $[X,Y]\in\mathcal{D}^{\prime}$. As $X\lrcorner
dU=-dx-ady$, one gets:%
\[
0=dU(X,Y)=-(dx+ady)([X,Y])=-[X,Y](x)-a[X,Y](y)=-\frac{X(a)}{a}\text{,}%
\]
so that $X(a)=0$, i.e., by Theorem \ref{Campi_Tipo2}, $X$ is of type $2$.
\end{proof}

\section{Normal forms of parabolic Monge-Amp\`{e}re equations}

\label{final_section}

In this section Theorems \ref{maintheor1} and \ref{maintheor2} are
eventually proved. Normal forms of parabolic MAE's are derived by
the corresponding normal forms of the associated characteristic
distributions. The relation between each normal form and the
existence of intermediate integrals is shown. Furthermore, the
existence of a complete integral for the general analytic
parabolic MAE's is proved.

\subsection{Intermediate integrals and their generalization}

\label{sec.inter.integrals}

\begin{definition}
Let $\mathcal{E}$ be a second order PDE. An \emph{intermediate integral} of
$\mathcal{E}$ is a function $f\in C^{\infty}(J^{1}(\tau))$ such that solutions
of the equations $f=k$, $k\in\mathbb{R}$, are also solutions of $\mathcal{E}$.
\end{definition}

In the case of MAE's, the following theorem provides a practical method for
finding intermediate integrals.

\begin{theorem}
[\cite{Alonso Blanco}]Let $\rho\in\Lambda^{2}(J^{1}(\tau))$ and $\mathcal{E}%
_{\rho}$ be the corresponding MAE. Then, $f\in C^{\infty}(J^{1}(\tau))$ is an
intermediate integral of $\mathcal{E}_{\rho}$ if and only if%
\begin{equation}
U\wedge df\wedge(X_{f}\rfloor\rho)=0. \label{Formula_Ric}%
\end{equation}

\end{theorem}

Coming back to the parabolic case, the following proposition holds.

\begin{proposition}
\label{integrale_intermedio} A function $f\in C^{\infty}(J^{1}(\tau))$ is an
intermediate integral of $\mathcal{E}_{\omega}$ if and only if $X_{f}%
\in\mathcal{D}$, i.e. $X_{f}$ is characteristic for the equation. Furthermore,
as $\mathcal{D}$ is lagrangian, $f$ is a first integral of any characteristic
field of $\mathcal{E}_{\omega}$.
\end{proposition}

\begin{proof}
If $\omega=X(U)\wedge Y(U)$, with $X$, $Y$ generating the characteristic
distribution of $\mathcal{E}_{\omega}$, then, by taking $\rho=\omega$ in
(\ref{Formula_Ric}), one gets%
\begin{equation}
U\wedge df\wedge W(U)=0 \label{FormulaRic2}%
\end{equation}
with $W=Y(f)X-X(f)Y$. But from (\ref{FormulaRic2}) follows $W(U)=\alpha
df+\beta U$ and, by dividing by $\alpha$, we obtain%
\begin{equation}
\frac{1}{\alpha}W(U)=df+\frac{\beta}{\alpha}U. \label{prima_espressione}%
\end{equation}
On the other hand $X_{f}(U)=df-f_{z}U$, so that, subtracting
(\ref{prima_espressione}) from it, one gets
\[
\left(  X_{f}-\frac{1}{\alpha}W\right)  (U)=\lambda U
\]
from which follows that $X_{f}-\frac{1}{\alpha}W=0$ (otherwise, it would be a
non-trivial characteristic field of $\mathcal{C}$), and the proposition follows.
\end{proof}

\begin{theorem}
\label{Goursat_Distributions}Let $\mathcal{D}\subset\mathcal{C}$ be the
characteristic distribution associated with $\mathcal{E}_{\omega}$. Then, such
equation admits intermediate integrals if and only if: 1) $\mathcal{D}$ is
integrable or 2) $\mathcal{D}^{\prime\prime}$ is $4$-dimensional and
integrable. In the first case, intermediate integrals are all and only the
functions of the form $f=\phi(f_{1},f_{2},f_{3})$ with $\phi$ arbitrary
function of three real variables and $f_{1},f_{2},f_{3}$ independent first
integrals of $\mathcal{D}$; in the second case, there exists (up to functional
dependence) only one intermediate integral, given by the function $f$ such
that $\mathcal{D}^{\prime\prime}=\{df=0\}$.
\end{theorem}

\begin{proof}
According to Proposition \ref{integrale_intermedio}, $f$ is an intermediate
integral if and only if $X_{f}\in\mathcal{D}$. If $\mathcal{D}$ is integrable,
then $\mathcal{D}=<X_{f_{1}},X_{f_{2}}>$ with $f_{1}$ and $f_{2}$ in
involution. Hence $X_{f_{1}}(f)=X_{f_{2}}(f)=0$ which proves the statement in
case 1).

If, instead, $\mathcal{D}$ is not integrable and $X_{f}\in\mathcal{D}$, then
$\mathcal{D}^{\prime}=X_{f}^{\bot}$ (see Proposition \ref{Unico_Tipo2}). It is
easily checked that $\mathcal{D}^{\prime\prime}=\{df=0\}$; in fact, two vector
fields are orthogonal to $X_{f}$ if and only if both have $f$ as a first
integral, so that their commutator vanishes on $df$.
\end{proof}

\medskip

It follows from the previous theorem that there exist parabolic MAE's without
intermediate integrals: in fact, as we shall see later, these are the
majority. For this reason, it is interesting to consider possible extensions
of the classical notion of intermediate integral. Note that a field $X$ is a
multiple of an $X_{f}$, with $f$ intermediate integral of $\mathcal{E}%
_{\omega}$, if and only if $X$ is a field of type $2$ in $\mathcal{D}$ such
that ($X^{\bot})^{\prime}$ is integrable. If one checks the last condition
out, one obtains \emph{nonholonomic intermediate integrals} in the sense of
\cite{KLR}.

\begin{definition}
\label{Int_Int_Non_Olo}Let $\mathcal{D}\subset\mathcal{C}$ be the
characteristic distribution associated with $\mathcal{E}_{\omega}$. A
\emph{nonholonomic intermediate integral} of $\mathcal{E}_{\omega}$ is a type
$2$ vector field contained in $\mathcal{D}$.
\end{definition}

\begin{theorem}
\label{Lychagin_Distributions}If $\mathcal{D}^{\prime\prime}$ is
$4$-dimensional, then $\mathcal{E}_{\omega}$ admits exactly one nonholonomic
intermediate integral $X\in\mathcal{D}$ which spans $(\mathcal{D}^{\prime
})^{\bot}$. Such an integral is classical if $\mathcal{D}^{\prime\prime}$ is
integrable and genuinely nonholonomic otherwise.
\end{theorem}

\begin{proof}
It is an easy corollary of Propositions \ref{Ortogonale_Tipo_2},
\ref{Unico_Tipo2} and Theorem \ref{Goursat_Distributions}.
\end{proof}

\smallskip Below we propose a further generalization.

\begin{definition}
\label{integrale_intermedio_generalizzato}A \emph{generalized intermediate
integral} of a parabolic MAE $\mathcal{E}_{\omega}$ is a field $X\in
\mathcal{D}$ of type less than $4$.
\end{definition}

Note that an intermediate integral of $\mathcal{E}_{\omega}$ is a
$4$-dimensional foliation of $J^{1}(\tau)$ whose leaves (which are first order
scalar differential equations) are such that their solutions are also
solutions of $\mathcal{E}_{\omega}$. By applying the method of
Lagrange-Charpit one obtains a complete integral ($2$ functional parameters)
of each leaf ($\infty^{1}$ leaves), so that one obtains a family of
$\infty^{3}$ solutions of $\mathcal{E}_{\omega}$.

\begin{definition}
\label{Integrale_Completo}A complete integral of $\mathcal{E}_{\omega}$ is a
2-dimensional foliation of $J^{1}(\tau)$ whose leaves are solutions or,
equivalently, a $2$-dimensional integrable distribution $\widehat{\mathcal{D}%
}\subset\mathcal{C}$ such that $\omega|_{\widehat{\mathcal{D}}_{\theta}}=0$
for any $\theta\in J^{1}(\tau)$.
\end{definition}

Let us now show the (almost) equivalence of the two above definitions.

\begin{proposition}
\label{Equivalenza_Intermedio_Completo}Starting from a generalized
intermediate integral, one can construct a complete integral, and viceversa.
\end{proposition}

\begin{proof}
If $X\in\mathcal{D}$ is of type $2$ or $3$, then it belongs to at least one
lagrangian integrable distribution $\widehat{\mathcal{D}}$ (Theorem
\ref{Teoremone}). Conversely, a complete integral $\widehat{\mathcal{D}}$,
whose fields are all of type $2$ or $3$, has a non trivial intersection with
$\mathcal{D}$: any non zero vector field in $\mathcal{D}\cap\widehat
{\mathcal{D}}$ is a generalized intermediate integral.
\end{proof}

\medskip

Note that the correspondence between intermediate integrals and complete
integrals is not biunivocal. Namely, when $X$ is of type $2$ it belongs to a
family of integrable distributions, whereas, when it is of type $3$ the
distribution is unique. Conversely, if $\dim\mathcal{D}\cap\widehat
{\mathcal{D}}=2$, i.e. $\mathcal{D}$ is integrable, then every field in
$\mathcal{D}$ is an intermediate integral; if, instead, $\dim\mathcal{D}%
\cap\widehat{\mathcal{D}}=1$, then the intermediate integral is unique (up to
a multiple). As we shall see in the next section, the latter is the generic case.

\subsection{The general case: proof of Theorem \ref{maintheor1}%
\label{subsec_Maintheor1}}

Let us assume assume that there exists a complete integral of $\mathcal{E}%
_{\omega}$. Then, by Proposition \ref{Equivalenza_Intermedio_Completo}, there
exists a generalized intermediate integral $Z\in\mathcal{D}$. As $Z$ is of
type less than $4$, by Theorem \ref{Teoremone} one has that, up to
contactomorphisms and a factor,
\[
Z=\partial_{p}+a\partial_{q}.
\]
Therefore, $\mathcal{D}$ is spanned by $Z$ and a vector field orthogonal to
it,
\[
W=\widehat{\partial}_{y}-a\widehat{\partial}_{x}+b\partial_{q},
\]
so that, up to a factor, is $\omega=Z(U)\wedge W(U)$, i.e.%
\[
\omega=-(dx+ ady)\wedge(dq-a dp-b dy)
\]
whose associated equation $\mathcal{E}_{\omega}$ is (\ref{Bryant_bis}), i.e.
\begin{equation}
z_{yy}-2a z_{xy}+a^{2}z_{xx}=b.\label{BG_form}%
\end{equation}
Viceversa, an equation of the above form
admits the characteristic field $Z=\partial_{p}+a\partial_{q}$
which belongs to the integrable distribution
$\widehat{\mathcal{D}}=<\partial _{p},\partial_{q}>$. This
completes the proof of Theorem \ref{maintheor1}.

\medskip The condition of the existence of a complete integral seems to be not
very restrictive in the $C^{\infty}$ category, as we shall see in section
\ref{subsubsec: A_rem_ex}. Furthermore we shall prove in section
\ref{subsubsec:The_anal_case} that, in the analytic case, this condition is
not a restriction at all.

\subsubsection{Does a complete integral always exist?}

\label{subsubsec: A_rem_ex}

Here we shall see how a large class of ($C^{\infty}$) parabolic MAE's admits a
complete integral and, hence, is reducible to normal form (\ref{Bryant_bis}).
Let us consider the parabolic MAE:
\[
z_{xy}^{2}-z_{xx}z_{yy}+Tz_{xx}-2Sz_{xy}+Rz_{yy}+S^{2}-RT=0
\]
which is associated with the distribution $\mathcal{D}$ spanned by vector
fields
\begin{equation}
\label{eq:X_and_Y_per la millesima_volta}X=\widehat{\partial}_{x}%
+R\partial_{p}+S\partial_{q}\,,\quad Y=\widehat{\partial}_{y}+S\partial
_{p}+T\partial_{q}%
\end{equation}
(see the end of section \ref{sec.par.MAE}). Assume either $R$ to be
independent of $q$ or $T$ to be independent of $p$. Then $\mathcal{D}$
contains a vector field of type $2$ or $3$. In fact, if $\partial_{q}(R)=0$,
then $[X,\partial_{q}]=-\partial_{q}(S)\partial_{q}$, so that the distribution
$<X,\partial_{q}>$ is integrable and the assertion follows from Theorem
\ref{Teoremone}. In the second case ($\partial_{p}(T)=0$) $Y$ belongs to the
integrable distribution $<Y,\partial_{p}>$.

\smallskip As an example, in order to give completely explicit computations,
we assume $R=1$. The distribution $<X,\partial_{q}>$ is integrable and spanned
by three common first integrals of the generators, namely:
\[
\left\langle X,\partial_{q}\right\rangle =\left\{  dy=d\alpha=d\beta
=0\right\}  \,,\quad\alpha=z-\frac{p^{2}}{2}, \quad\beta=x-p.
\]
Then $\{y=k_{1},\,\alpha=k_{2},\,\beta=k_{3}\}$, $k_{i}\in\mathbb{R}$, turns
out to be a complete integral of the MAE under consideration. A direct
computation shows that $U=d\alpha-p\,d\beta-q\,dy$. Therefore, functions
\[
\overline{x}=\beta=x-p,\;\;\overline{y}=y,\;\;\overline{z}=\alpha
=z-\frac{p^{2}}{2},\;\;\overline{p}=p,\;\;\overline{q}=q
\]
are contact coordinates, with respect to which $X$ and $Y$ are given by%
\[
X=\partial_{\overline{p}}+S\partial_{\overline{q}}\,, \quad Y=\widehat
{\partial}_{\overline{y}}-S\widehat{\partial}_{\overline{x}}+S\partial
_{\overline{p}}+T\partial_{\overline{q}}.
\]
Since $\partial_{\overline{p}}=X-S\partial_{\overline{q}}$, $\mathcal{D}$ is
spanned by%
\[
X=\partial_{\overline{p}}+S\partial_{\overline{q}},\, \quad Y^{\prime
}=\widehat{\partial}_{\overline{y}}-S\widehat{\partial}_{\overline{x}%
}+(T-S^{2})\partial_{\overline{q}}
\]
and the associated equation becomes
\[
\overline{z}_{\overline{y}\overline{y}}-2S\overline{z}_{\overline{x}%
\overline{y}}+S^{2}\overline{z}_{\overline{x}\overline{x}}-(T-S^{2})=0.
\]

\subsubsection{The analytic case}

\label{subsubsec:The_anal_case}

In \cite{Bryant Griffiths} it is proved that every parabolic MAE with real
analytic coefficients can be reduced to form (\ref{Bryant_bis}) by means of
Cartan-K\"{a}hler theorem. In this section we give an alternative proof based
only on the Cauchy-Kovalevsky theorem.

\smallskip As we already explained, all that we have to do is to find a
complete integral. As a first step, we give some equivalent formulations of
this problem without yet assuming the analyticity condition.

\begin{lemma}
A vector field $Z\in\mathcal{C}$ is of type less than $4$ if and only if it
admits a first integral $f$ satisfying the equation%
\begin{equation}
\label{Equazione_Infernale}dU(Z,[Z,X_{f}])=0\,,\,\,\, \text{with}\,\,\,
X_{f}\neq0.
\end{equation}

\end{lemma}

\begin{proof}
If $Z$ is a multiple of $X_{f}$ for some $f$, then both of them are of type
$2$. So, we can assume that they are independent. It is easy to prove that if
$Z(f)=0$ then $dU(X_{f},[Z,X_{f}])=0$. Assume that the first integral $f$ is a
solution of (\ref{Equazione_Infernale}); then $[Z,X_{f}]$ is orthogonal to the
lagrangian distribution spanned by $Z$ and $X_{f}$ and, hence, belongs to it;
but this implies that such distribution is integrable. By applying Theorem
\ref{Teoremone} one obtains that $Z$ is of type $2$ or $3$.

Conversely, if $Z$ is of type $2$ or $3$ then, again by Theorem
\ref{Teoremone}, $Z$ linearly depends on two fields $X_{f}$, $X_{g}$ with $f$
and $g$ in involution: obviously, both functions are solutions of
(\ref{Equazione_Infernale}).
\end{proof}

\begin{theorem}
\label{Teoremone2}Let $\mathcal{D}=<X,Y>$ be the lagrangian distribution
associated with equation $\mathcal{E}_{\omega}$. Then, the following
equivalences hold:

\begin{itemize}
\item[1)] There exists a complete integral of $\mathcal{E}_{\omega}$;

\item[2)] There exists a generalized intermediate integral;

\item[3)] There exists a field $Z\in\mathcal{D}$ such that type $Z<4$;

\item[4)] There exists a field $Z\in\mathcal{D}$ which is also contained in an
integrable lagrangian distribution $\widehat{\mathcal{D}}$;

\item[5)] There exists an integrable lagrangian distribution $\widehat
{\mathcal{D}}$ such that the graph of the corresponding section $J^{1}%
(\tau)\rightarrow J^{2}(\tau)$ is contained in $\mathcal{E}_{\omega}$.

\item[6)] There exists a function $f\in C^{\infty}(J^{1}(\tau))$ such that the
field $Z_{f}=Y(f)X-X(f)Y$ satisfies the equation%
\begin{equation}
\label{eq:equazione_impossibile}dU(Z_{f},[Z_{f},X_{f}])=0;
\end{equation}

\end{itemize}
\end{theorem}

\begin{proof}
The equivalence of properties $1)$, $2)$, $3)$, $4)$, $5)$ has been already
proved. Let us focus on the equivalence between $4)$ and $6)$. First, $4)$
implies $6)$. In fact, let us suppose $Z\in\widehat{\mathcal{D}}$. Since
$\widehat{\mathcal{D}}$ is integrable, there exists a function $f$ such that
$X_{f}\in\widehat{\mathcal{D}}$ (see Theorem \ref{th.sostiutivo}) and
$Z(f)=0$, which implies that $Z$ is a multiple of $Z_{f}$. So $Z_{f}%
,[Z_{f},X_{f}]\in\widehat{\mathcal{D}}$, that is lagrangian, and
(\ref{eq:equazione_impossibile}) follows. Second, $6)$ implies $4)$. If
$Z_{f}=0$, then $X(f)=Y(f)=0$, which implies $X_{f}\in\mathcal{D}$. Then we
can choose $Z=X_{f}$. If $Z_{f}\neq0$, then it is sufficient to apply previous
lemma with $Z=Z_{f}$.
\end{proof}

\smallskip The determining equation (\ref{eq:equazione_impossibile}) provides
a tool for proving the existence of a complete integral in the real analytic case.

\begin{theorem}
\label{Better_Than_Bryant}Any parabolic analytic MAE admits a complete
integral. In particular, it can be reduced to form (\ref{Bryant_bis}).
\end{theorem}

\begin{proof}
Equation (\ref{eq:equazione_impossibile}) can be written in the equivalent
form:
\begin{equation}
Y(f)^{2}dU(X,[X,X_{f}])-2X(f)Y(f)dU(X,[Y,X_{f}])+X(f)^{2}dU(Y,[Y,X_{f}])=0.
\label{eq:equazione_impossibile_2}%
\end{equation}

It is straightforward to check that this equation, in a contact chart where
$X$ and $Y$ assume the form (\ref{eq:X_and_Y_per la millesima_volta}), takes
the form
\begin{equation}
\label{eq:local_expr_of_eq_imp}\sum_{i,j=1}^{5} A^{ij}f_{x^{i}x^{j}}+ B=0,
\end{equation}
where we have denoted by $(x^{1},x^{2},x^{3},x^{4},x^{5})$ the chart
$(x,y,z,p,q)$, and $A^{ij}$ and $B$ are analytic functions of $x^{1}%
,\dots,x^{5},f_{x^{1}},\dots,f_{x^{5}}$. Hence, by applying Cauchy-Kovalevsky
theorem to equation (\ref{eq:local_expr_of_eq_imp}), the existence of a
complete integral in a neighborhood of an arbitrary analytic hypersurface of
$J^{1}(\tau)$ is proved.
\end{proof}

\subsection{The non generic case: proof of Theorem \ref{maintheor2}%
\label{subsec_Maintheor2}}

In the previous section (proof of Theorem \ref{maintheor1}) we derived the
normal form (\ref{BG_form}) of a parabolic MAE admitting a complete integral
from that  of the associated characteristic distribution:%
\[
\mathcal{D}=<\partial_{p}+a\partial_{q}\,,\,\widehat{\partial}%
_{y}-a\widehat{\partial}_{x}+b\partial_{q}>
\]

As we have already seen, such canonical form holds for all analytic parabolic
MAE's and for a large class of $C^{\infty}$ ones (indeed, we strongly suspect,
for all). In particular, one can reduce to form (\ref{BG_form}) all
\textit{non generic} parabolic MAE's, i.e. those for which $\mathcal{D}%
^{\prime\prime}$ has dimension less than $5$. However, for such equations more
precise normal forms can be obtained.

\begin{theorem}
\label{Forme_Normali_Distribuzioni} Let $\mathcal{D}\subset\mathcal{C}$ be a
non generic lagrangian distribution. Then, there exist contact local
coordinates on $J^{1}(\tau)$ in which $\mathcal{D}$ takes one the following
normal forms:

\begin{itemize}
\item[a)] $\mathcal{D}=<\widehat{\partial}_{x},\,\partial_{q}>$;

\item[b)]
$\mathcal{D}=<\partial_{p},\,\widehat{\partial}_{y}+b\partial_{q}>$,\,\,
$b\in C^{\infty}(J^{1}(\tau))$, \,\, $\partial_{p}(b)\neq0$;

\item[c)]
$\mathcal{D}=<\partial_{p}+z\partial_{q},\,\widehat{\partial}_{y}
-z\widehat{\partial}_{x}+b\partial_{q}>$, \,\, $b\in
C^{\infty}(J^{1}(\tau))$, \,\,
$\partial_{p}(b)+z\partial_{z}(b)\neq0$.
\end{itemize}
\end{theorem}

\begin{proof}
According to the \textquotedblleft integrability degree" of $\mathcal{D}$, one
can distinguish the following cases:

\begin{itemize}
\item[1)] $\mathcal{D}=\mathcal{D}^{\prime}$, i.e. $\mathcal{D}$ is integrable;

\item[2)] $\mathcal{D}\neq\mathcal{D}^{\prime}$, i.e. $\mathcal{D}$ is non
integrable: in this case $\dim\mathcal{D}^{\prime}=3$ and $\mathcal{D}%
^{\prime}\subset\mathcal{C}$ (the latter property is due to the fact that
$\mathcal{D}$ is lagrangian); by Proposition \ref{3_mai_integrabile},
$\mathcal{D}^{\prime}$ is non integrable.

Case 2) splits into the following subcases:

\begin{itemize}
\item[2-1)] $\mathcal{D}\neq\mathcal{D}^{\prime}\neq\mathcal{D}^{\prime\prime
}$ and $\dim\mathcal{D}^{\prime\prime}=4$; in this case there are two possibilities:

\begin{itemize}
\item[2-1-1)] $\mathcal{D}^{\prime\prime}$ integrable;

\item[2-1-2)] $\mathcal{D}^{\prime\prime}$ non integrable;
\end{itemize}

\item[2-2)] the generic case: $\mathcal{D}\neq\mathcal{D}^{\prime}%
\neq\mathcal{D}^{\prime\prime}$ and $\dim\mathcal{D}^{\prime\prime}=5$.
\end{itemize}
\end{itemize}

\begin{itemize}
\item In case 1), in view of Theorem
\ref{Caratterizzazione_Integrabilita}, in a suitable contact chart
$\mathcal{D}$ takes the form
\[
\mathcal{D}=<\partial_{p},\partial_{q}>
\]
and, by a Legendre transformation, we obtain normal form $a)$.
\item In case 2), $\mathcal{D}^{\prime}$ is determined by a
generator $X$ of its orthogonal complement. Let us examine, first,
case 2-1). From Theorem \ref{Lychagin_Distributions}, and in view
of Theorem \ref{Campi_Tipo2}, one obtains the normal form for the
field $X\in(\mathcal{D}^{\prime})^{\bot}$:

\item the case 2-1-1) corresponds to the normal form $X=\partial_{p}$, so that
we obtain normal form $b)$.

\item the case 2-1-2) corresponds to the normal form $X=\partial_{p}
+z\partial_{q}$, so that we obtain normal form $c)$.

\item The case 2-2) is excluded by hypothesis.
\end{itemize}
\end{proof}

\medskip

Note that it is possible to distinguish the various types of parabolic MAE's
according to the number and kind of their intermediate integrals, namely:
\begin{itemize}
\item[-] in case 1) there are three intermediate integrals, up to
functional dependence, and according to Theorem \ref{Teoremone}
$\mathcal{D}$ contains only vector fields of type less than $4$;

\item[-] in case 2-1-1) there exists only one intermediate
integral and, in view of Proposition \ref{Unico_Tipo2}, only one
vector field of type $2$ which turns out to be hamiltonian;

\item[-] in case 2-1-2) there are no classical intermediate
integrals, but there exists a nonholonomic one in the sense of
\cite{KLR}, which is also, up to a factor, the only vector field
of type $2$ (Proposition \ref{Unico_Tipo2}).

\item[-] in case 2-2) there is not even a nonholonomic integral.
For what said in the previous section, there exists a generalized
intermediate integral (fields of type $3$) in the real analytic
case, while we don't know in the $C^{\infty}$ case.
\end{itemize}

\smallskip
In order to obtain normal of Theorem \ref{maintheor2} by using the
results of previous theorem, it is sufficient to compute
$\mathcal{E}_\omega$ where $\omega=X(U)\wedge Y(U)$ with
$\mathcal{D}=<X,Y>$ (see also the reasoning in the end of section
\ref{sec.par.MAE}). This completes the proof of Theorem
\ref{maintheor2}.


\bigskip
\noindent\textbf{Acknowledgement.} The authors thank A.M.
Vinogradov for drawing their attention to the notion of type of a
Cartan field. The first author thanks J. Mu\~{n}oz, A.
\'{A}lvarez, S. Jim\'{e}nez and J. Rodr\'{\i}guez for many useful
suggestions and encouragements. The second author thanks the
Department of Mathematics \textquotedblleft Ennio De Giorgi" for
financial support.

\end{document}